\documentclass[12pt]{amsart}
\usepackage{amssymb}
\usepackage{verbatim}
\usepackage{hyperref}
\usepackage{color}

\begin{document}

\newtheorem{theorem}{Theorem}    
\newtheorem{proposition}[theorem]{Proposition}
\newtheorem{conjecture}[theorem]{Conjecture}
\def\theconjecture{\unskip}
\newtheorem{corollary}[theorem]{Corollary}
\newtheorem{lemma}[theorem]{Lemma}
\newtheorem{sublemma}[theorem]{Sublemma}
\newtheorem{fact}[theorem]{Fact}
\newtheorem{observation}[theorem]{Observation}
\theoremstyle{definition}
\newtheorem{definition}{Definition}
\newtheorem{notation}[definition]{Notation}
\newtheorem{remark}[definition]{Remark}
\newtheorem{question}[definition]{Question}
\newtheorem{questions}[definition]{Questions}
\newtheorem{example}[definition]{Example}
\newtheorem{problem}[definition]{Problem}
\newtheorem{exercise}[definition]{Exercise}

\numberwithin{theorem}{section}
\numberwithin{definition}{section}
\numberwithin{equation}{section}

\def\reals{{\mathbb R}}
\def\torus{{\mathbb T}}
\def\heis{{\mathbb H}}
\def\integers{{\mathbb Z}}
\def\rationals{{\mathbb Q}}
\def\naturals{{\mathbb N}}
\def\complex{{\mathbb C}\/}
\def\distance{\operatorname{distance}\,}
\def\support{\operatorname{support}\,}
\def\dist{\operatorname{dist}\,}
\def\Span{\operatorname{span}\,}
\def\degree{\operatorname{degree}\,}
\def\kernel{\operatorname{kernel}\,}
\def\dim{\operatorname{dim}\,}
\def\codim{\operatorname{codim}}
\def\trace{\operatorname{trace\,}}
\def\Span{\operatorname{span}\,}
\def\dimension{\operatorname{dimension}\,}
\def\codimension{\operatorname{codimension}\,}
\def\nullspace{\scriptk}
\def\kernel{\operatorname{Ker}}
\def\ZZ{ {\mathbb Z} }
\def\p{\partial}
\def\rp{{ ^{-1} }}
\def\Re{\operatorname{Re\,} }
\def\Im{\operatorname{Im\,} }
\def\ov{\overline}
\def\eps{\varepsilon}
\def\lt{L^2}
\def\diver{\operatorname{div}}
\def\curl{\operatorname{curl}}
\def\etta{\eta}
\newcommand{\norm}[1]{ \|  #1 \|}
\def\expect{\mathbb E}
\def\bull{$\bullet$\ }
\def\det{\operatorname{det}}
\def\Det{\operatorname{Det}}
\def\multiR{\mathbf R}
\def\bestA{\mathbf A}
\def\Apq{\mathbf A_{p,q}}
\def\Apqr{\mathbf A_{p,q,r}}
\def\rank{\mathbf r}
\def\diameter{\operatorname{diameter}}
\def\bp{\mathbf p}
\def\bff{\mathbf f}
\def\bg{\mathbf g}
\def\essd{\operatorname{essential\ diameter}}
\def\mab{\max(|A|,|B|)}

\newcommand{\abr}[1]{ \langle  #1 \rangle}

\newcommand{\Norm}[1]{ \Big\|  #1 \Big\| }
\newcommand{\set}[1]{ \left\{ #1 \right\} }
\def\one{{\mathbf 1}}
\newcommand{\modulo}[2]{[#1]_{#2}}

\def\scriptf{{\mathcal F}}
\def\scriptq{{\mathcal Q}}
\def\scriptg{{\mathcal G}}
\def\scriptm{{\mathcal M}}
\def\scriptb{{\mathcal B}}
\def\scriptc{{\mathcal C}}
\def\scriptt{{\mathcal T}}
\def\scripti{{\mathcal I}}
\def\scripte{{\mathcal E}}
\def\scriptv{{\mathcal V}}
\def\scriptw{{\mathcal W}}
\def\scriptu{{\mathcal U}}
\def\scriptS{{\mathcal S}}
\def\scripta{{\mathcal A}}
\def\scriptr{{\mathcal R}}
\def\scripto{{\mathcal O}}
\def\scripth{{\mathcal H}}
\def\scriptd{{\mathcal D}}
\def\scriptl{{\mathcal L}}
\def\scriptn{{\mathcal N}}
\def\scriptp{{\mathcal P}}
\def\scriptk{{\mathcal K}}
\def\scriptP{{\mathcal P}}
\def\scriptj{{\mathcal J}}
\def\frakv{{\mathfrak V}}
\def\frakG{{\mathfrak G}}
\def\frakA{{\mathfrak A}}
\def\frakB{{\mathfrak B}}
\def\frakC{{\mathfrak C}}

\author{Michael Christ}
\address{
        Michael Christ\\
        Department of Mathematics\\
        University of California \\
        Berkeley, CA 94720-3840, USA}
\email{mchrist@math.berkeley.edu}
\thanks{Research supported in part by NSF grant DMS-0901569 and by the Mathematical Sciences Research Institute.}


\date{December 20, 2011}

\title {Near--Extremizers of Young's Inequality for $\reals^d$}
\begin{abstract}
Any pair of functions which nearly extremizes Young's convolution inequality
$\norm{f*g}_t\le A\norm{f}_p\norm{g}_q$ for $\reals^d$, 
is close in norm to a pair which is an exact extremizer.
\end{abstract}
\maketitle

\section{Introduction}
Young's convolution inequality for $\reals^d$ states that if $(p,q)\in[1,\infty]$,
if $\rho^{-1}=p^{-1}+q^{-1}-1$, and if $\rho\in[1,\infty]$ then
\begin{equation}
\norm{f*g}_{\rho}\le \norm{f}_p\norm{g}_q 
\end{equation}
where $\norm{\cdot}_s$ denotes the Lebesgue $L^s(\reals^d)$ norm.
The convolution product is $f*g(x)=\int_{\reals^d} f(x-y)g(y)\,dy$.

The inequality holds in this form for any locally compact Abelian group if Haar measure is used
to define the convolution and $L^s$ norms.
But for $\reals^d$, it holds in a sharper form, established by Beckner \cite{beckner} and Brascamp-Lieb \cite{brascamplieb}:
\begin{equation} \label{lessyoung} \norm{f*g}_{\rho}\le {\mathbf A}_{p,q}^d\norm{f}_p\norm{g}_q \end{equation}
where
${\mathbf A}_{p,q} = C_pC_qC_{\rho'}$ with $C_p^2=\frac{p^{1/p}}{(p')^{1/p'}}$.
This number, which is strictly less than one,
is the optimal constant. Moreover, equality is realized only if there exist an invertible affine
endomorphism $L$ of $\reals^d$ and vectors $u,v\in\reals^d$ such that
$f(x)$ is a nonzero scalar multiple of $e^{-|L(x-u)|^2}$
and 
$g(x)$ is a nonzero scalar multiple of $e^{-\gamma|L(x-v)|^2}$
where $\gamma$ is a positive constant which depends only on $p,q$ \cite{beckner},\cite{brascamplieb}.
In this situation we say that $(f,g)$ is an extremizing pair for the inequality \eqref{lessyoung}.

In this paper we characterize pairs of functions $(f,g)\in L^p\times L^q$ which nearly realize this
optimal constant, demonstrating that near equality can only occur for functions which are close
in norm to extremizers.

\begin{theorem}\label{thm:main}
Let $d\ge 1$. Let $p,q\in(1,\infty)$, and suppose that the exponent $\rho$ defined by
$\rho^{-1}=p^{-1}+q^{-1}-1$ belongs to $(1,\infty)$. 
For any $\eps>0$ there exists $\delta>0$ such that
for any $\complex$--valued pair $(f,g)\in (L^p\times L^q)(\reals^d)$ of functions with nonzero norms,
if 
\begin{equation}\norm{f*g}_{\rho}\ge (1-\delta){\mathbf A}_{p,q}^d\norm{f}_p\norm{g}_q\end{equation} 
then there exists an extremizing pair $(F,G)$ such that
\begin{equation} \norm{f-F}_p\le\eps\norm{f}_p \text{ and } \norm{g-G}_q\le\eps\norm{g}_q.\end{equation} 
\end{theorem}

This conclusion is false if one or more of the three exponents equal $1$ or $\infty$.
Eisner and Tao \cite{eisnertao} have shown that for arbitrary locally compact Abelian groups, if the
optimal constant equals $1$, then all
near-extremizers are close to scalar multiples of indicator functions of cosets of compact open subgroups. 
For discrete groups, the optimal constant in the inequality does equal $1$.
It is shown in \cite{youngdiscrete} that for any discrete group with no nontrivial finite subgroups,
near-extremizers must be close to extremizers.  However, these analyses do not apply when the
optimal constant is strictly less than $1$.

The crux of the matter is the one-dimensional case, for nonnegative functions. 
The extension to complex-valued functions and to higher dimensions follows from supplementary arguments
which begin in \S\ref{section:extensions}. 

Our analysis of the one-dimensional case relies on an approximate inverse rearrangement inequality.
Write $\langle f,h\rangle = \int_{\reals^d}fg$.
Denote by $f^\star$ the radially symmetric 
nonincreasing rearrangement of a nonnegative function $f$ with domain $\reals^d$.
The rearrangement inequality of Riesz \cite{riesz} and Sobolev \cite{sobolev} 
states that for nonnegative $f,g,h$,
$\langle f*g,h\rangle \le\langle f^\star*g^\star,h^\star\rangle$.
See \cite{liebloss} for a proof and discussion.
For dimension $d=1$,
Burchard \cite{burchard} proved that if $f,g,h$ are indicator functions of
sets with positive, finite measures satisfying certain natural inequalities,
and if equality holds in the Riesz-Sobolev inequality for these functions,
then the three sets in question must differ from intervals by null sets. 

The approximate inverse rearrangement inequality \cite{christrieszsobolev} used here also applies to
indicator functions of sets, with a hypothesis of near equality in the form
$\langle f*g,h\rangle \ge(1-\eps)\langle f^\star*g^\star,h^\star\rangle$
and a correspondingly weaker conclusion of the form
$|A\bigtriangleup I|\le\delta |A|$ for some interval $I\subset\reals^1$,
where $f=\one_A$, where $\delta$ is a function of $\eps$ which tends to $0$ as $\eps$ tends to $0$.
Here $A\bigtriangleup B$ denotes the symmetric difference of two sets.
A weakness of this result is that it presupposes near equality
not for a single indicator function $h=\one_S$, but rather, for two such functions. But as will be 
seen below, our analysis leads naturally to a situation in which near equality is known to hold
for infinitely many sets $S$, in a strong sense.

The following notational conventions will be used throughout the paper. 
We will consistently use the quantity $\delta$ to quantify the degree to which
a pair or triple of functions nearly extremizes Young's inequality.
For other quantities $\eta$ which arise in the discussion, the notation $\eta=o_\delta(1)$
will mean that $\eta\le \psi (\delta)$ where $\psi(\delta)\to 0$ as $\delta\to 0$.
The function $\psi$ may depend on the dimension $d$, 
and sometimes on other parameters which will be indicated,
but otherwise will depend only on
$\delta$, perhaps through auxiliary quantities chosen in the course of the proof which themselves
depend only on $\delta,d$.
An ordered triple $(f,g,h)$ is said to be close to an ordered triple $(\tilde f,\tilde g,\tilde h)$
in $L^p\times L^q\times L^r$ if $\norm{f-\tilde f}_p$ is small and if
$g,h$ are likewise close to $\tilde g,\tilde h$ in $L^q,L^r$ respectively.
A triple $(f,g,h)$ is said to be $\delta$--close to $(\tilde f,\tilde g,\tilde h)$ in $L^p\times L^q\times L^r$
if $\norm{f-\tilde f}_p<\delta$, $\norm{g-\tilde g}_q<\delta$, and $\norm{h-\tilde h}_r<\delta$.

The author thanks Terence Tao for a stimulating conversation. 

\section{Normalized extremizing sequences}

Let $p,q,r$ satisfy 
\begin{align}
\label{H1}
p,q,r\in(1,\infty)
\\
\label{H2}
p^{-1}+q^{-1}+r^{-1}=2.
\end{align}
These hypotheses will be in force throughout the discussion.
Let $r'=r/(r-1)$ be the exponent conjugate to $r$.

Let $d\ge 1$.
It will be convenient to write Young's inequality in the more symmetric form
\begin{equation} \big|\langle f*g,h\rangle\big| \le \Apqr^d\norm{f}_p\norm{g}_q\norm{h}_r \end{equation}
where
$\Apqr=\Apq = C_pC_qC_r$, where $C_s^2= \frac{s^{1/s}}{t^{1/t}}$ with $t=s'$ the exponent 
conjugate to $s$.

\begin{lemma} \label{lemma:comparablemeasures}
There exist $C<\infty$ and $\gamma>0$, depending on $(p,q,r)$, such that
for any measurable sets $E,E'\subset\reals^d$ with positive and finite Lebesgue measures,
\begin{equation} \norm{\one_E*\one_{E'}}_{r'}\le C\min\Big(\frac{|E|}{|E'|},\frac{|E'|}{|E|}\Big)^\gamma |E|^{1/p}|E'|^{1/q}. 
\end{equation}
\end{lemma}
\begin{proof} Trivial.\end{proof}

Let $f$ be a measurable function which is finite almost everywhere.
There exist a unique decomposition
\begin{equation} f= \sum_{j\in\integers}  2^j F_j\end{equation}
and associated pairwise disjoint sets $\scriptf_j$ satisfying
\begin{equation} \one_{\scriptf_j}\le |F_j|<2\one_{\scriptf_j}.  \end{equation}
For $g,h$ there are corresponding decompositions in terms of functions $G_j,H_j$ associated to sets $\scriptg_j,\scripth_j$,
respectively.

\begin{lemma}\label{lemma:crudelevelsetdecomp}
There exist finite positive constants  $c_0,C_0$ and positive functions $\theta,\Theta$
such that $\theta(t)\to 0$ as $t\to \infty$ and $\Theta(\delta)\to\infty$ as $\delta\to 0$, with the following properties. 
Let $(f,g,h)\in (L^p\times L^q\times L^r)(\reals^d)$ satisfy $\norm{f}_p=\norm{g}_q=\norm{h}_r=1$
and $\langle f*g,h\rangle \ge (1-\delta)\Apqr^d$.  Let $F_j,\scriptf_j,G_j,\scriptg_j,H_j,\scripth_j$ 
be the functions and sets associated to $f,g,h$ respectively.
Then there exist $k,k',k''\in\integers$  such that
\begin{align} 2^k|\scriptf_k|^{1/p}&\ge c_0 \\
2^j|\scriptf_j|^{1/p}&\le\theta(|j-k|)  \text{ whenever $|j-k|\le \Theta(\delta)$}
\end{align} 
with corresponding properties for $\set{\scriptg_j}$ relative to $k'$
and for $\set{\scripth_j}$ relative to $k''$.
Moreover
\begin{equation} |k-k'|+|k-k''|\le C_0.  \end{equation}
\end{lemma}

The triple $(f,g,h)$ is therefore equivalent, under
the action of the affine symmetry group, to a pair with $k=0$ and $|k'|+|k''|\le C$. 
The lemma is a consequence of Lemma~\ref{lemma:comparablemeasures} via the  reasoning in  
\S5 and Lemma~6.1 of \cite{christparaboloidconvolutionextremizers}, so the proof is omitted. \qed

\begin{definition} 
Let $p,q,r\in(1,\infty)$.  Let $\Theta,R$ be functions which satisfy 
$\Theta(\rho)\to 0$ as $\rho\to\infty$ and $R(t)\to\infty$ as $t\to 0$.  
A function $f\in L^p$ is said to be $\eta$--normalized with exponent $p$
with respect to $\Theta,R$ if $\norm{f}_p=1$ and
\begin{align} &\int_{|\tilde f(x)|> \rho} \tilde f(x)^p\,dx\le\Theta(\rho) \text{ for all $\rho\le R(\eta)$}
\label{normalized1} \\
&\int_{|\tilde f(x)|<\rho^{-1}} \tilde f(x)^p\,dx\le\Theta(\rho) \text{ for all $\rho\le R(\eta)$}.
\label{normalized2} \end{align}
\end{definition}
We will most often omit mention of the exponent $p$ when referring to this definition.
Given an ordered triple of exponents $(p,q,r)$,
we will say that a triple of functions $(f,g,h)$ is $\eta$--normalized with
respect to $\Theta,R$ if each of the three functions is normalized with the corresponding exponent.

\begin{definition}
An extremizing sequence for Young's inequality for an ordered triple of exponents $(p,q,r)$
is a sequence of elements $(f_\nu,g_\nu,h_\nu)\in L^p\times L^q\times L^r$
such that $\norm{f_\nu}_p\equiv\norm{g_\nu}_q\equiv\norm{h_\nu}_r\equiv 1$
and $\langle f_\nu*g_\nu,h_\nu\rangle \to \Apqr$.
\end{definition}

\begin{definition}
Let $(p,q,r)$ satisfy \eqref{H1},\eqref{H2}.  An extremizing sequence $(f_\nu,g_\nu,h_\nu)$
is said to be normalized if there exist functions $\Theta,R$  and a sequence $(\eta_\nu)$
of positive real numbers satisfying $\lim_{\nu\to\infty} \eta_\nu=0$, 
$\Theta(\rho)\to 0$ as $\rho\to\infty$, and $R(t)\to\infty$ as $t\to 0$, 
such that for each index $\nu$,
the functions $f\_nu,g_\nu,h_\nu$ are $\eta_\nu$--normalized with respect to $\Theta,R$
with exponents $p,q,r$ respectively.
\end{definition}

In these terms, Lemma~\ref{lemma:crudelevelsetdecomp} can be restated thusly:
\begin{proposition} \label{prop:normalized}
Let $p,q,r$ satisfy \eqref{H1},\eqref{H2}.
There exist functions $\Theta,R$ satisfying $\Theta(\rho)\to 0$ as $\rho\to\infty$
and $R(t)\to\infty$ as $t\to 0$ with the following property.
Let $\delta>0$ be arbitrary.
Let $(f,g,h)\in L^p\times L^q\times L^r$ be any $(1-\delta)$--nearly extremizing triple satisfying
$\norm{f}_p= \norm{g}_q= \norm{h}_r=1$.  There exists $\lambda\in\reals^+$  such that the functions 
$\tilde f(x)=\lambda^{1/p}f(\lambda (x))$, $\tilde g(x)=\lambda^{1/q}g(\lambda  (x) )$, 
$\tilde h(x)=\lambda^{1/r}h(\lambda  (x))$
are $\delta$--normalized with respect to the functions $\Theta,R$, with exponents $p,q,r$ respectively.
\end{proposition}

We will prove below:
\begin{proposition} \label{prop:nonnegativecase} 
Let $(p,q,r)$ satisfy \eqref{H1},\eqref{H2}.
For any $\eps>0$ there exists $\delta>0$ with the following property.
If $(f,g,h)\in (L^p\times L^q\times L^r)(\reals^1)$ is a $(1-\delta)$--near extremizer
for Young's inequality which is $\delta$--normalized,
then there exist $a,b\in\reals$ such that 
$(f(x-a),g(x-b),h(x-a-b))$ is $\eps(\delta)$--close to an extremizing triple.
\end{proposition}
\noindent
By $f(x-a)$ we mean of course the function $\reals^1\owns x\mapsto f(x-a)$.

\section{Control of distribution functions}

Consider any normalized near-extremizing triple $(f,g,h)$.
The purpose of this section is to establish
strong control over the distribution functions $|\set{x: f(x)>t}|$,
with corresponding control for $g,h$.
We will do this by proving
that the triple of symmetric rearrangements $(f^\star,g^\star,h^\star)$
is close in $L^p\times L^q\times L^r$ to some triple which is an exact extremizer.
The discussion will exploit the following characterization of extremizers \cite{beckner},\cite{brascamplieb}. 
A generalization, with a different method of proof, is established in \cite{BCCT1}.   
\begin{proposition}\label{prop:beckneretal}
For fixed $(p,q,r)$, there exist $\sigma,\tau>0$ such that any extremizing ordered triple
$(\scriptf,\scriptg,\scripth)$ is of the form
$\scriptf(x)=c_1 \exp(-\lambda (x-a_1)^2)$,
$\scriptg(x)=c_2 \exp(-\sigma \lambda (x-a_2)^2)$,
$\scripth(x)=c_3 \exp(-\tau \lambda (x-a_3)^2)$
for some $\lambda>0$, $a_j\in\reals$, and $0\ne c_j\in\complex$ satisfying $c_1c_2c_3>0$ and $a_3=a_1+a_2$.
Conversely, any such triple is an extremizer.
\end{proposition}

In particular, extremizing triples are unique, up to scalar multiplication
and the action of the natural affine group; they are bounded; their nonincreasing rearrangements
are continuous and strictly decreasing.
Only these qualitative facts about the extremizers are actually needed in our analysis.

\begin{proposition} \label{prop:starprecompact}
Let $\gamma\in(0,1]$ and $\Gamma\in[1,\infty)$.
Let $(f_\nu,g_\nu,h_\nu)\in (L^p\times L^q\times L^r)(\reals^d)$ be any extremizing sequence satisfying
$\norm{f_\nu}_p= \norm{g_\nu}_q= \norm{h_\nu}_r=1$ and 
\begin{equation} \label{starnormalized} f_\nu^\star(1),g_\nu^\star(1),h_\nu^\star(1)\in[\gamma,\Gamma]. \end{equation}
Then the sequence of rearrangements $f_\nu^\star$
is precompact in $L^p$. Likewise, the sequences $g_\nu^\star,h_\nu^\star$ are precompact in $L^q,L^r$ respectively.
\end{proposition}

\begin{proof}
Together, the assumption that $f_\nu^\star(1)$ remains bounded above and below,
and the normalization hypotheses \eqref{normalized1} and \eqref{normalized2},
imply that the associated scalars $\lambda_\nu$ 
provided by Proposition~\ref{prop:normalized} remain bounded above and below by positive constants.
Therefore for any $\eps>0$ there exist  $\alpha>0$ and $\beta<\infty$
such that for all sufficiently large indices $\nu$,
\begin{equation} \label{alphabetanegligible}
\int_{|t|\le\alpha} f_\nu^\star(t)^p\,dt + \int_{|t|\ge\beta} f_\nu^\star(t)^p\,dt < \eps.
\end{equation}
Corresponding conclusions hold for $g_\nu^\star,h_\nu^\star$ with $p$ replaced by $q,r$ respectively.

Now $2\alpha f_\nu^\star(\alpha)^p\le\norm{f_\nu^\star}_p^p=1$,
so $f_\nu^\star(\alpha)$ remains uniformly bounded.
Since $f_\nu^\star$ is also monotonic nonincreasing along rays, 
the sequence $f_\nu^\star$ is precompact in $L^p$ of the annulus
$\set{x\in\reals^d: \alpha\le |x|\le \beta}$.
Inequality \eqref{alphabetanegligible} then implies precompactness in $L^p(\reals^d)$.
The same reasoning applies to the sequences $(g_\nu^\star),(h_\nu^\star)$. 
\end{proof}

Any normalized sequence $(f_\nu,g_\nu,h_\nu)$ satisfies the hypothesis \eqref{starnormalized}, 
so Proposition~\ref{prop:starprecompact} applies to it. Denote again by 
$(f_\nu,g_\nu,h_\nu)$ such a subsequence, chosen so that the sequence of rearrangements
$(f_\nu^\star,g_\nu^\star,h_\nu^\star)$ is convergent in $L^p\times L^q\times L^r$.
Since $\langle f_\nu^\star*g_\nu^\star,h_\nu^\star\rangle \ge \langle f_\nu*g_\nu,h_\nu\rangle$,
the sequence $(f_\nu^\star,g_\nu^\star,h_\nu^\star)$ is likewise extremizing. Its 
limit $(\scriptf,\scriptg,\scripth)$ is an extremizing triple. 
By Proposition~\ref{prop:beckneretal}, each of $\scriptf,\scriptg,\scripth$ is equal to
a Gaussian $c e^{-\lambda x^2}$ with norm equal to $1$
in $L^p,L^q,L^r$ respectively. Moreover, the parameters $c,\lambda$ both lie in some compact subinterval of $(0,\infty)$,
by virtue of the two normalizations $f_\nu^\star(1)\in[\gamma,\Gamma]$ and $1=\norm{f_\nu^\star}_p$.

Given $(f,g,h)$, denote the corresponding superlevel sets by
\[F_{s}=\set{x\in\reals: f(x)>s},\ G_{s}=\set{x\in\reals: g(x)>s},\ H_{s}=\set{x\in\reals: h(x)>s},\]
along with corresponding superlevel sets $\scriptf_{s},\scriptg_{s},\scripth_{s}$ for $\scriptf,\scriptg,\scripth$
and $F^\star_{s},G^\star_{s},H^\star_{s}$ for the rearrangements $f^\star,g^\star,h\star$, respectively.  

\begin{lemma}  \label{lemma:comparesuperlevelsets}
Consider any $\delta$--normalized ordered triple
$(f,g,h)$ of nonnegative functions such that $(f^\star,g^\star,h^\star)$
is $\delta$--close to an extremizing triple $(\scriptf,\scriptg,\scripth)$
in $L^p\times L^q\times L^r$. Then
\begin{equation} \label{measuresconverge}
\big|\,|F_{s}|-|\scriptf_{s}|\,\big|=o_\delta(1) \end{equation}
uniformly for all $s$ in any fixed compact subset of $(0,\infty)$.  Similarly 
$\big|\,|G_{s}|-|\scriptg_{s}|\,\big|=o_\delta(1) $ and $\big|\,|H_{s}|-|\scripth_{s}|\,\big|=o_\delta(1) $,
with the same uniformity.  \end{lemma}

\begin{proof} By Chebychev's inequality,
\[ |F_{s}^\star\setminus \scriptf_{s-\eta}|\le \eta^{-p}\norm{f^\star-\scriptf}_p^p \]
and likewise
\[ |\scriptf_{s+\eta}\setminus F_{s}^\star|\le \eta^{-p}\norm{f^\star-\scriptf}_p^p. \]
For $s$ in any compact subset of $(0,\norm{\scriptf}_\infty)$,
$|\scriptf_{s-\eta}\setminus\scriptf_{s+\eta}|\le C\eta$.
Therefore by choosing a function $\delta\mapsto\eta(\delta)$ such that $\eta(\delta)\to 0$
sufficiently slowly as $\delta\to 0$, we can conclude that
$\big|\,|F_{s}^\star|-|\scriptf_{s}|\,\big|=o_\delta(1)$.  But $|F_{s}|=|F_{s}^\star|$.

The same reasoning applies to $\big|\,|G_{s}|-|\scriptg_s| \big|$ and $\big|\,|H_{s}|-|\scripth_s| \big|$.
\end{proof}

\section{Superlevel sets are nearly intervals}

In the following statement, $F_\alpha$ continues to denote the superlevel set $F_\alpha=\set{x: f(x)>\alpha}$,
and likewise $\scriptf_\alpha,\scriptg_\beta$ denote the superlevel sets of functions $\scriptf,\scriptg$.

\begin{lemma} \label{lemma:nearlyintervals}
Fix $(p,q,r)$ satisfying \eqref{H1},\eqref{H2} and functions $\Theta,R$ with the properties discussed above.
For any $\varrho>0$ there exists a function $\delta\mapsto\eps(\delta)$ 
satisfying $\eps(\delta)\to 0$ as $\delta\to 0$ with the following property.
Let the ordered triple of nonnegative functions $(f,g,h)\in (L^p\times L^q\times L^r)(\reals^1)$  be
$(1-\delta)$--nearly extremizing for Young's inequality, and be $\delta$--normalized with respect to $\Theta,R$.
Let $(\scriptf,\scriptg,\scripth)\in L^p\times L^q\times L^r$ be an exactly extremizing ordered
triple of nonnegative functions, and 
suppose that the triple of  symmetric nonincreasing rearrangements $(f^\star,g^\star,h^\star)$
is $\delta$--close to $(\scriptf,\scriptg,\scripth)$ in $L^p\times L^q\times L^r$.
Then for every $\alpha$ for which there exists $\beta$ such that 
\begin{align} \label{eq:comparablemeasureshypothesis}
 \max(|\scriptf_\alpha|,|\scriptg_\beta|)&\le (2-\varrho) \min(|\scriptf_\alpha|,|\scriptg_\beta|)
\\ \varrho\le\alpha&\le \norm{\scriptf}_\infty-\varrho
\\ \varrho\le \beta&\le \norm{\scriptg}_\infty-\varrho, \end{align}
there exists an interval $I_\alpha\subset\reals$ such that
\begin{equation}\big|I_\alpha\bigtriangleup F_{\alpha}\big|=o_\delta(1).  \end{equation}
\end{lemma}

\begin{proof}
Decompose
\begin{equation}
\langle f*g,h\rangle
= \iiint_{(0,\infty)^3} 
\langle \one_{F_{\alpha}}*\one_{G_{\beta}},\,\one_{H_{\gamma}} \rangle
\,d\alpha\,d\beta\,d\gamma.
\end{equation}
By the Riesz-Sobolev inequality,
\begin{equation}
\langle \one_{F_{\alpha}}*\one_{G_{\beta}},\,\one_{H_{\gamma}} \rangle
\le
\langle \one_{F^\star_{\alpha}}*\one_{G^\star_{\beta}},\,\one_{H^\star_{\gamma}} \rangle
\end{equation}
for all $\alpha,\beta,\gamma$.

The first step in the proof is the following simple fact.
\begin{lemma} \label{lemma:onlymediumeta}
\begin{equation}
\lim_{\eta\to 0}\  \limsup_{\delta\to 0}\ \sup_f\ 
\Big\|\int_{\alpha\notin[\eta,\norm{\scriptf}_\infty-\eta]} \one_{F_{\alpha}}\,d\alpha \Big\|_p = 0.
\end{equation}
Corresponding statements hold for $g,h$ with exponents $q,r$ respectively.
\end{lemma}
\noindent
The notation $\sup_f$ means that the supremumn is taken over all functions $f$ which are
first components of ordered triples $(f,g,h)$ which satisfy the hypotheses of Lemma~\ref{lemma:nearlyintervals}. 

\begin{proof}  
Firstly,
\begin{align*}
\norm{\int_0^{\eta}\one_{F_\alpha}\,d\alpha}_p = \norm{\min(f^\star,\eta)}_p
\le \norm{f^\star-\scriptf}_p + \norm{\min(\scriptf,\eta)}_p
=o_\delta(1) + o_{\eta}(1).  \end{align*}
Secondly,
\begin{align*}
\norm{\int_{\norm{\scriptf_\infty}-\eta}^{\infty} \one_{F_\alpha}\,d\alpha}_p
&\le \norm{\max(0,f^\star-(\norm{\scriptf}_\infty-\eta))}_p
\\
&\le \norm{\max(0,\scriptf-(\norm{\scriptf}_\infty-\eta))}_p + \norm{f^\star-\scriptf}_p
\\
&= o_{\eta}(1)+o_\delta(1).
\end{align*}
\end{proof}

Define
\begin{equation} \Omega =\Omega(\eta)= [\eta,\norm{\scriptf}_\infty-\eta]\times
[\eta,\norm{\scriptg}_\infty-\eta]\times [\eta,\norm{\scripth}_\infty-\eta] \subset(0,\infty)^3.  \end{equation}
A consequence of the Lemma~\ref{lemma:onlymediumeta} and Young's inequality is that for any $\eta>0$,
\begin{equation} \label{eq:mediumetaconsequence} 
\lim_{\eta\to 0}\  \limsup_{\delta\to 0}\ \sup_{(f,g,h)}\ 
\iiint_{(0,\infty)^3\setminus\Omega(\eta)} 
\langle \one_{\scriptf_\alpha}*\one_{\scriptg_\beta},\,\one_{\scripth_\gamma}\rangle
\,d\alpha\,d\beta\,d\gamma  = 0,
\end{equation}
with the supremum taken over all triples $(f,g,h)$ which satisfy the hypotheses of Lemma~\ref{lemma:nearlyintervals}
for given $\delta$.

Next, for any $\eta>0$,
\begin{equation}  
\iiint_{\Omega(\eta)} \langle \one_{F_{\alpha}}*\one_{G_{\beta}},\,\one_{H_{\gamma}} \rangle
\,d\alpha\,d\beta\,d\gamma 
\ge \big(1-o_{\delta}(1)\big) \iiint_{\Omega(\eta)} 
\langle \one_{\scriptf_\alpha}*\one_{\scriptg_\beta},\,\one_{\scripth_\gamma}\rangle
\,d\alpha\,d\beta\,d\gamma \end{equation}
by Lemma~\ref{lemma:comparesuperlevelsets}.
The superlevel sets are nested; $F_{\alpha}\subset F_{\alpha'}$ whenever $\alpha\ge\alpha'$.
From this nesting property, the comparison of measures of superlevel sets \eqref{measuresconverge},
and the fact that $\alpha\mapsto |\scriptf_\alpha|$ is a continuous, strictly decreasing function of $\alpha$
on the interval $[0,\norm{\scriptf}_\infty]$, and by the corresponding properties of $g,h$,
it follows that for any fixed $\eta$,
\begin{equation}\label{eq:ptwisenearlyextremal} \langle \one_{F_{\alpha}}*\one_{G_{\beta}},\,\one_{H_{\gamma}} \rangle
\ge (1-o_\delta(1))\langle \one_{\scriptf_\alpha}*\one_{\scriptg_\beta},\,\one_{\scripth_\gamma}\rangle
\end{equation}
uniformly for all $(\alpha,\beta,\gamma)\in\Omega(\eta)$.

The purpose of this discussion has been to reach a position from which it is possible to invoke
an inverse theorem established in \cite{christrieszsobolev}, which states the following. 
\begin{theorem} \label{thm:inverse}
Let $\rho>0$. Let $A,B\subset\reals$ be measurable sets with finite, positive measures
satisfying $\max(|A|,|B|)\le (2-\rho)\min(|A|,|B|)$.
Suppose that there exist $t\ge \rho\mab$ satisfying $3t\le (1-\rho)(|A|+|B|)$
and measurable sets $E,E'$ satisfying
$\big|\,|E|-t\,\big|\le \tau\mab$ and $\big|\,|E'|-3t\,\big|\le \tau\mab$
such that
\begin{align}
\langle \one_A*\one_B,\one_{E}\rangle &\ge \langle \one_{A^\star}*\one_{B^\star},\one_{E^\star}\rangle -\tau\mab^2
\\
\langle \one_A*\one_B,\one_{E'}\rangle &\ge \langle \one_{A^\star}*\one_{B^\star},\one_{E'^\star}\rangle -\tau\mab^2.
\end{align}
Then there exists an interval $J$ such that $|A\bigtriangleup J|<o_\tau(1) |A|$.
\end{theorem}
\noindent The factor denoted $o_\tau(1)$ does depend on $\rho$.
However, it follows that there exists a function $\tau\mapsto\rho(\tau)$ 
satisfying $\lim_{\tau\to 0}\rho(\tau)=0$ such that the 
conclusion still holds if the hypotheses are satisfied with $\rho=\rho(\delta)$. 
We will apply Theorem~\ref{thm:inverse} in this form.

To prove Lemma~\ref{lemma:nearlyintervals}, consider any $\alpha,\beta$ satisfying the 
hypotheses of the lemma.
Because $\scripth$ is a Gaussian, its superlevel sets 
$\scripth_\gamma=\set{x: \scripth(x)>\gamma}$ have measures which take on all values in
the range $(0,\infty)$ as $\gamma$ varies over $(0,\norm{\scripth}_\infty)$.
Given $t$, there exist unique $\gamma,\gamma'$ satisfying $|\scripth_\gamma|=t$ and $|\scripth_{\gamma'}|=3t$. 
Apply Theorem~\ref{thm:inverse} with $A=F_{\alpha}$, $B=G_{\beta}$, and $E=H_{\gamma}$, $E'=H_{\gamma'}$
for these parameters $\gamma,\gamma'$.

We next verify that the hypotheses of the inverse theorem are satisfied for some $\tau=o_\delta(1)$.
Assume that $(\alpha,\beta,\gamma)\in\Omega(\eta(\delta))$. 
Firstly, 
\[\big|\,|E_t|-t\,\big| =\big|\,|H_{\gamma}|-|\scripth_\gamma|\, \big|=o_\delta(1)\]
by \eqref{measuresconverge},
uniformly for all $\gamma$ in any compact subinterval of $(0,\norm{\scripth}_\infty)$
and therefore uniformly for all $t$ in any compact subinterval of $(0,\infty)$.

Secondly, continuing to define $\gamma$ by $|\scripth_\gamma|=t$,
\begin{align*}\langle \one_A*\one_B,\one_{E_t}\rangle 
= \langle \one_{F_{\alpha}}*\one_{G_{\beta}},\,\one_{H_{\gamma}} \rangle
&\ge (1-o_\delta(1))\langle \one_{\scriptf_\alpha}*\one_{\scriptg_\beta},\,\one_{\scripth_\gamma}\rangle
\\
&\ge \langle \one_{\scriptf_\alpha}*\one_{\scriptg_\beta},\,\one_{\scripth_\gamma}\rangle
-o_\delta(1) \max(|F_{\alpha}|,|G_{\beta}|)^2.
\end{align*}
The second inequality holds because $\big|\,|F_{\alpha}|-|\scriptf_\alpha|\,\big|=o_\delta(1)$. 

We are assuming that
$ \max(|\scriptf_\alpha|,|\scriptg_\beta|)\le (2-\varrho) \min(|\scriptf_\alpha|,|\scriptg_\beta|)$.
Again since 
$\big|\,|F_{\alpha}|-|\scriptf_\alpha|\,\big|=o_\delta(1)$ and 
$\big|\,|G_{\beta}|-|\scriptg_\beta|\,\big|=o_\delta(1)$, 
it follows that 
\[\max(|F_{\alpha}|,|G_{\beta}|)\ge (2-\varrho-o_\delta(1)) \min(|F_{\alpha}|,|G_{\beta}|).\]
Consequently Theorem~\ref{thm:inverse} applies, and guarantees the existence of the interval $I$ promised in the 
statement of Lemma~\ref{lemma:nearlyintervals}.
\end{proof}

Consider any $\alpha\in [\varrho, \norm{\scriptf}_\infty-\varrho]$. 
As $\beta$ varies over $(0,\norm{\scriptg}_\infty)$, $|\scriptg_\beta|$ 
takes on all values in $(0,\infty)$, so there certainly exists $\beta$ such that
$\max(|\scriptf_\alpha|,|\scriptg_\beta|)\le (2-\varrho)\min(|\scriptf_\alpha|,|\scriptg_\beta|)$.
Therefore we may apply Lemma~\ref{lemma:nearlyintervals} to conclude that $F_{\alpha}$ nearly coincides with 
some interval. 
Therefore we have proved:
\begin{lemma} \label{lemma:gotintervals}
Suppose that $\Theta,R$  satisfy (??). For any $\eta>0$
there exists $\delta>0$ with the following property.
Let $(f,g,h)\in (L^p\times L^q\times L^r)(\reals^1)$ 
be any $(1-\delta)$--nearly extremizing ordered triple of nonnegative functions which is normalized
with respect to $\Theta,R$.
Suppose that $(f^\star,g^\star,h^\star)$ is $\delta$--close in $L^p\times L^q\times L^r$ 
to an exact extremizing triple $(\scriptf,\scriptg,\scripth)$ of nonnegative functions.  Then
for each $\alpha\in [\eta,\norm{\scriptf}_\infty-\eta]$ there exists an interval $I_\alpha$ such that 
\begin{equation} |F_{\alpha} \bigtriangleup I_\alpha|<o_\delta(1)|\scriptf_\alpha|.\end{equation}
This bound holds uniformly for all $\alpha\in[\eta,\norm{\scriptf}_\infty-\eta]$.
\end{lemma}
\noindent
By symmetry, corresponding intervals exist for the sets $G_{\beta}$ and $H_{\gamma}$.
It would be equivalent to write $|F_{\alpha} \bigtriangleup I_\alpha|<o_\delta(1)|F_\alpha|$
in the conclusion since it has been shown that $|F_\alpha-\scriptf_\alpha|=o_\delta(1)$
in this range of parameters. 

\section{Precompactness}

In order to establish Theorem~\ref{thm:main} for nonnegative functions in dimension $d=1$,
it suffices to prove the following.
\begin{proposition} \label{prop:precompactness}
Let $(p,q,r)$ satisfy \eqref{H1},\eqref{H2}.
Let $(f_\nu,g_\nu,h_\nu)$ be a normalized extremizing sequence of ordered triples of nonnegative
functions in $(L^p\times L^q\times L^r)(\reals^1)$.
Suppose that
$(f_\nu^\star,g_\nu^\star,h_\nu^\star)$ converges in $L^p\times L^q\times L^r$.
Then there exist sequences $a_\nu,b_\nu$ of real numbers such that the sequence of ordered triples
$(f_\nu(x-a_\nu),g_\nu(x-b_\nu),h_\nu(x-a_\nu-b_\nu))$ is precompact in $L^p\times L^q\times L^r$.
\end{proposition}
\noindent Here $F(x-a)$ denotes the function $\reals^1\owns x\mapsto F(x-a)$.

\begin{proof} 
Let $F_{\nu,\alpha},G_{\nu,\beta},H_{\nu,\gamma}$ denote the superlevel sets of $f_\nu,g_\nu,h_\nu$, respectively.

Let $\eta>0$. Apply Lemma~\ref{lemma:gotintervals} to obtain intervals $I_{\nu,\alpha},J_{\nu,\beta},K_{\nu,\gamma}$
associated as in the conclusion of that lemma to the sets $F_{\nu,\alpha},G_{\nu,\beta},H_{\nu,\gamma}$ respectively
for all $(\alpha,\beta,\gamma)\in\Omega(\eta)$. Thus 
$\big| F_{\nu,\alpha}\bigtriangleup I_{\nu,\alpha}\big|<\eps_\nu$,
where $\eps_\nu=\eps_\nu(\eta) \to 0$ as $\nu\to\infty$, uniformly in $\alpha$ so long as $\eta$ remains fixed. 
Corresponding conclusions hold for $G_{\nu,\beta}\bigtriangleup J_{\nu,\beta}$
and for $H_{\nu,\gamma}\bigtriangleup K_{\nu,\gamma}$.

Consider first the sequence of functions $f_\nu = \int_0^\infty \one_{F_{\nu,\alpha}}\,d\alpha$.
By Lemma~\ref{lemma:onlymediumeta}, 
in order to show that the sequence $(f_\nu)$ is precompact in $L^p$ modulo translations,
it suffices to prove that for any $\eta>0$, the sequence of functions
$\int_\eta^{\norm{\scriptf}_\infty-\eta} \one_{F_{\nu,\alpha}}\,d\alpha$, indexed by $\nu$, is precompact 
modulo translations. 
Since $|F_{\nu,\alpha}\bigtriangleup I_{\nu,\alpha}|\to 0$ uniformly
for all $\alpha$ in this interval as $\nu\to\infty$, 
this is equivalent to the $L^p$ precompactness modulo translations of the sequence of functions
\begin{equation} \label{precompact!} \int_\eta^{\norm{\scriptf}_\infty-\eta} \one_{I_{\nu,\alpha}}\,d\alpha.\end{equation}

\begin{lemma} Let $\eta>0$. Suppose that the interval $I_{\nu,\eta}$ is centered at $0$ for each index $\nu$.
Then the sequence of functions \eqref{precompact!} is precompact in $L^p(\reals^1)$.  \end{lemma}

\begin{proof}
Whenever $\alpha\ge\eta$, $F_{\nu,\alpha}\subset F_{\nu,\eta}$.  Therefore 
\begin{align*}\big|I_{\nu,\alpha}\setminus I_{\nu,\eta} \big| 
&\le |I_{\nu,\alpha}\bigtriangleup F_{\nu,\alpha}| + |I_{\nu,\eta}\bigtriangleup F_{\nu,\eta}|
+|F_{\nu,\alpha}\setminus F_{\nu,\eta}| 
\\ &\le \eps_\nu+\eps_\nu+0. \end{align*}
Thus the $L^p$ norm of the restriction of 
$\int_\eta^{\norm{\scriptf}_\infty-\eta} \one_{I_{\nu,\alpha}}\,d\alpha$ 
to the complement of $I_{\nu,\eta}$ tends to $0$ as $\nu\to\infty$.

It remains to analyze
$\int_\eta^{\norm{\scriptf}_\infty-\eta} \one_{I_{\nu,\alpha}\cap I_{\nu,\eta}}\,d\alpha$. 
As $\nu\to\infty$, $|I_{\nu,\eta}|\to |\scriptf_\eta|<\infty$, so the intervals $I_{\nu,\eta}$
remain in a bounded subset of $\reals$. 
Finally, since $I_{\nu,\alpha}\cap I_{\nu,\eta}$ is an interval for each $\alpha$,
each function \eqref{precompact!}
is of total variation $\le 2\norm{\scriptf}_\infty$, 
has $L^\infty$ norm not exceeding $\norm{\scriptf_\infty}$,  
and is supported in a bounded interval independent of $\nu$.
Therefore the sequence of functions \eqref{precompact!} is precompact in $L^p(\reals)$.
\end{proof}

The same reasoning applies to the sequences $(g_\nu)$ and $(h_\nu)$. 
Therefore under the hypotheses of 
Proposition~\ref{prop:precompactness}, after passing to a subsequence of the original sequence $(f_\nu,g_\nu,h_\nu)$,
there exist sequences $a_\nu,b_\nu,c_\nu$ of real numbers
such that the sequence of functions $(f_\nu(x-a_\nu))$ of $x\in\reals$ is convergent in $L^p$,
and likewise $(g_\nu(x-b_\nu)),(h_\nu(x-c_\nu))$ converge $L^q,L^r$ respectively.
To complete the proof of Proposition~\ref{prop:precompactness}, it remains only to observe that since
$(f_\nu,g_\nu,h_\nu)$ is normalized and extremizing, the convergence of these sequences forces
$c_\nu-(a_\nu+b_\nu)$ to remain bounded as $\nu\to\infty$.  
Passage to a further subsequence ensures that the sequence of functions $h_\nu(x-a_\nu-b_\nu)$ converges in $L^r$.
\end{proof}

\section{An approximate functional equation }\label{section:keylemma}
Any measurable function $\varphi:\reals^d\to\complex$ which satisfies
$\varphi(x)+\varphi(y)\equiv \varphi(x+y)$ is linear.
Here we establish an approximate version of this fact, which will be the key to extending
the result proved thus far to higher dimensions. 
The lemmas of this section are likely to have other applications.

For any $R>0$, denote by $B_R$ the ball $B_R=\set{x\in\reals^d: |x|<R}$.
\begin{proposition}\label{prop:nearlylinear}
For each dimension $d\ge 1$ 
there exist a constant $C(d)<\infty$ and a function $\delta\to\eps(\delta)$ satisfying $\eps(\delta)\to 0$
with the following property.
Let $B=B_R$ be any ball in $\reals^d$ with positive radius, and let $\tau>0$.
Let $f:B_{2R}\to\complex$ be a measurable function. Suppose that 
\begin{equation} \big|\set{(x,y)\in B^2: |f(x)+f(y)-f(x+y)|>\tau}\big|<\delta |B|^2.  \end{equation}
Then there exists a linear function $L:\reals^d\to\complex$ such that
\begin{equation} \big|\set{x\in B: |f(x)-L(x)|>C(d)\tau}\big|<\eps(\delta) |B|.  \end{equation}
\end{proposition}

The same proof will show that if the hypothesis is strengthened to $f(x)+f(y)=f(x+y)$
on the complement of a set of measure $\delta|B|^2$, then the conclusion  can be strengthened to
\begin{equation} \big|\set{x\in B: f(x)\ne L(x)}\big|<\eps |B|.  \end{equation}

There is no hypothesis here that $f$ be locally integrable, let alone
satisfy some upper bound. This precludes the use of certain otherwise natural analytic techniques.  

A multiplicative analogue of Proposition~\ref{prop:nearlylinear} is equally natural and will 
be used to extend Theorem~\ref{thm:main} to complex-valued functions.
\begin{proposition}\label{prop:nearlycharacter}
For each dimension $d\ge 1$ 
there exist a constant $C(d)<\infty$ and a function $\delta\to\eps(\delta)$ satisfying $\eps(\delta)\to 0$
with the following property.
Let $B=B_R$ be any ball in $\reals^d$ with positive radius, and let $\tau>0$.
Let $f:B_{2R}\to\complex$ be a measurable function which vanishes only on a set of measure zero. Suppose that 
\begin{equation} \big|\set{(x,y)\in B^2: |f(x)f(y)f(x+y)^{-1}-1|>\tau}\big|<\delta |B|^2.  \end{equation}
Then there exists a linear function $L:\reals^d\to\complex$ such that 
\begin{align} \big|\set{x\in B: |f(x)e^{-L(x)}|>C\tau}\big| &<\eps(\delta) |B|  \text{ if $\tau\le 1$}  
\\
\big|\set{x\in B: |f(x)e^{-L(x)}|>C\tau}^C\big| &<\eps(\delta) |B|  \text{ if $\tau\ge 1$.}  \end{align}
\end{proposition}

In the context of the proof of Proposition~\ref{prop:nearlylinear}, we will utilize the following notations and definitions.
\begin{definition}  Let $f:B_{2R}\to\complex$ be measurable. Let $\delta,\tau$ be the parameters
given in the statement of Proposition~\ref{prop:nearlylinear}. Let $\gamma,\lambda,\sigma>0$.
\begin{itemize}
\item
$a\approx_\lambda b$ will mean that $|a-b|<\lambda$.  
\item
If $T\subset S$ are subsets of some measure space,
we say that the vast majority of all points in $S$ belong to $T$ if $|T|\ge (1-o_\delta(1))|S|$.
\item
Define 
\begin{equation} \scriptf_\sigma =\set{(x,y)\in B_R^2: |f(x+y)-f(x)-f(y)|\le \sigma}.  \end{equation}
\item
Let $\gamma>0$.  A point $x\in B_R$ is $\gamma$--rich if $(x,y)\in \scriptf_\tau$ for all $y\in B_R^2$ 
in the complement of a set of measure $\gamma |B_R|$.
\end{itemize}
\end{definition}

\begin{proof}[Proof of Proposition~\ref{prop:nearlylinear}]
Every $x\in B_R$ is $C\delta^{1/2}$--rich, with the exception of a set of measure $\le C\delta^{1/2}|B_R|$.
Therefore for all $(x,y)\in B_{R/2}^2$ except a set of measure
$\le C\delta^{1/2}|B_R|^2$,  $x,y,x+y$ will all be $C\delta^{1/2}$--rich.

If $(x,y,w,z)\in B_{R/2}^4$, and if each of the ordered pairs
$(x,w),(y,z),(w,z),(x+w,y+z),(x+y,w+z)$ belongs to $\scriptf_\tau$,
then $(x,y)\in\scriptf_{5\tau}$. 
Indeed,
\begin{equation*} f(x+w)+f(y+z)\approx_{2\tau} f(x)+f(w)+f(y)+f(z) \approx_\tau f(x)+f(y)+f(w+z) \end{equation*}
while on the other hand
\begin{equation*} f(x+w)+f(y+z)\approx_\tau f(x+w+y+z)\approx_\tau f(x+y)+f(w+z). \end{equation*}

If $x,y$ and $x+y$ are all $\gamma$--rich for sufficiently small $\gamma$,
then for the vast majority of all $(w,z)\in B_{R/2}^2$,
each of the five ordered pairs listed will indeed belong to $\scriptf_{\tau}$.
Fix such a parameter $\gamma$.  Hence we say that $z\in B_R$ is rich if it is $\gamma$--rich for this value of $\gamma$.

We have proved the following lemma.
\begin{lemma} \label{lemma:rich}
If $\delta$ is sufficiently small, then $|f(x+y)-f(x)-f(y)|<5\tau$ whenever $(x,y)\in B_{R/2}^2$
and $x,y,x+y$ are all rich.  \end{lemma}

\begin{lemma}
If $\delta>0$ is sufficiently small then
there exists an absolute constant $C<\infty$ such that whenever
$x_j\in B_{R/8}$ are rich and satisfy $x_1-x_2+x_3-x_4=0$, 
\[|f(x_1)-f(x_2)+f(x_3)-f(x_4)|<C\tau.\]
\end{lemma}

\begin{proof}
Observe that if $y_j\in B_{R/8}$ have the property that each of the quantities
$y_j$, $x_j+y_j$, $x_1+x_3+y_1+y_3$, $x_2+x_4+y_2+y_4$,  $x_1+x_3+y_1+y_3 -x_2-x_4-y_2-y_4$,
$y_1-y_2$, and $y_3-y_4$ is rich,
then the following chain of approximate equalities is justified by Lemma~\ref{lemma:rich}:
\begin{align*}
f(x_1)-f(x_2)&+f(x_3)-f(x_4)
\\
&\approx_{20\tau} \sum_{j=1}^4 (-1)^{j-1} \big(f(x_j+y_j)-f(y_j)\big)
\\
&=\big( f(x_1+y_1) + f(x_3+y_3)\big)
-\big( f(x_2+y_2) + f(x_4+y_4)\big)
+\sum_{j=1}^4 (-1)^{j-1}f(y_j)
\\
&\approx_{10\tau} f(x_1+x_3+y_1+y_3)
-f(x_2+x_4+y_2+y_4)
+\sum_{j=1}^4 (-1)^{j-1}f(y_j)
\\
&\approx_{5\tau} f(x_1+x_3+y_1+y_3 -x_2-x_4-y_2-y_4)
+\sum_{j=1}^4 (-1)^{j-1}f(y_j)
\\
&=f(y_1-y_2+y_3-y_4)+\sum_{j=1}^4 (-1)^{j-1}f(y_j)
\\
&\approx_{15\tau} 0.
\end{align*}
If $\delta$ is sufficiently small then the vast majority of all $4$-tuples $(y_1,y_2,y_3,y_4)\in B_{R/8}^4$ 
have the required properties.  Thus we have proved that if $u,v\in B_{R/8}$ are rich then 
$f(u)+f(v)$ is approximately well defined, in the sense that
if $u',v'\in B_{R/8}$ are also rich and if $u'+v'=u+v$,
then $f(u')+f(v')$ differs from $f(u)+f(v)$ by no more than $C\tau$.
\end{proof}

Now consider any $z\in B_{R/8}$. 
Because the vast majority of all elements of $B_{R/8}$ are rich, there exist rich $u,v\in B_{R/8}$
satisfying $u+v=z$. The associated sum $f(u)+f(v)$ then depends only on $z$, rather than on the choice of $u,v$,
up to an additive error which is $O(\tau)$.
For $z\in B_{R/8}$ define $\varphi(z)$ to be the average of $f(u)+f(z-u)$,
averaged with respect to Lebesgue measure over all $u\in B_{R/8}^2$ such that $u,z-u$ 
both belong to $B_{R/8}$ and are both rich.  The function $\varphi$ is measurable.
For every rich $z\in B_{R/8}$,
$f(z)\approx_{C\tau}\varphi(z)$ since $f(u)+f(v)\approx_{\tau} f(u+v)$ for an overwhelming majority of all pairs $u,v$.
Therefore $f\equiv \varphi+O(\tau)$ on $B_{R/8}$, except on a set of measure $\eps(\delta)|B|$,
where $\eps(\delta)\to 0$ as $\delta\to 0$.

\begin{lemma}
Suppose that $\delta$ is sufficiently small. Then
for almost every pair $(z_1,z_2)\in B_{R/8}^2$, 
$\varphi(z_1)+\varphi(z_2)$ depends only on $z_1+z_2$ up to an additive error which is $O(\tau)$.
\end{lemma}

\begin{proof}
Write $z_1=x_1+y_1$ and $z_2=x_2+y_2$ where $x_j,y_j\in B_{R/8}$ are all rich,
$\varphi(z_1)-f(x_1)-f(y_1)=O(\tau)$, and $\varphi(z_2)-f(x_2)-f(y_2)=O(\tau)$.
Then
\begin{align*}
\varphi(z_1)+\varphi(z_2) &= f(x_1)+f(y_1)+f(x_2)+f(y_2)+O(\tau)
\\ &=f(x_1+x_2)+f(y_1+y_2)+O(\tau)
\\ &=f(x_1+x_2+y_1+y_2)+O(\tau)
\\ &=f(z_1+z_2)+O(\tau)
\end{align*}
provided that $x_1+x_2$, $y_1+y_2$ are rich, which is the case for the
overwhelming majority of choices of $x_1,x_2$.
\end{proof}

Extend the definition of $\varphi$ to $B_{2R}$ by defining $\varphi(z)$
to be the average of $\sum_j \varphi(x_j)$,
with the average taken over all representations of $z$ as $\sum_{j=1}^{32}x_j$
with each $x_j\in B_{R/8}$, with respect to Lebesgue measure on the hyperplane
in $\reals^{32d}$ defined by $\sum_j x_j=z$.
It follows by a simpler analogue of reasoning given above that
\begin{equation} \varphi(x)+\varphi(y)=\varphi(x+y)+O(\tau) \text{ for all $(x,y)\in B_R^2$.} \end{equation}
In the same way, 
\begin{equation} \label{phinearlyadditive}
\varphi(\sum_{i=1}^N x_i)=\sum_{i=1}^N \varphi(x_i)+O(N\tau)
\end{equation}
for all $x_i\in B_{R}$, for any fixed $N$ provided that $\delta$ is sufficiently small.
In particular, $\varphi(0)=O(\tau)$ and $\varphi(-x)\equiv -\varphi(x)+O(\tau)$.

Consider now the special case $d=1$. Define $\psi = \varphi-L$ where $L:\reals^d\to\complex$
is the unique linear function such that $L(R)=\varphi(R)$.

Define $\Psi:\reals/R\integers\to\complex$ by $\Psi(x)=\psi(\tilde x)$
where $x$ denotes a coset in $\reals/R\integers$, and $\tilde x\in[0,R)$ denotes a representative of that coset.  Then 
\begin{equation} \label{Psinearlyadditive} \Psi(x)+\Psi(y)=\Psi(x+y)+O(\tau) \end{equation}
for all $x,y\in\reals/R\integers$.
Indeed, let $\tilde x,\tilde y\in[0,R)$ be representatives of the cosets $x,y$.  If $\tilde x+\tilde y<R$
then \eqref{Psinearlyadditive} is a consequence of the corresponding property of $\varphi$.
If $\tilde x+\tilde y\in[R,2R)$ then write
$\tilde x+\tilde y = \big(\tilde x+(R-\tilde x)\big) + \big( \tilde x+ \tilde y-R \big)$
and apply \eqref{phinearlyadditive} repeatedly together with the relation $\psi(R)=0$
to obtain \eqref{Psinearlyadditive}.

We claim that $\norm{\Psi}_{L^\infty(\reals/R\integers)}=O(\tau)$.
Indeed, for any $x\in\reals/R\integers$, 
average the relation $\Psi(x)=\Psi(x+y)-\Psi(y)+O(\tau)$ with respect to $y\in \reals/R\integers$.
The averages of $\Psi(x+y)$ and $\Psi(y)$ cancel, leaving $\Psi(x)=O(\tau)$.

Now this means that $\varphi(x)=L(x)+O(\tau)$ for all $x\in [0,R]$, and the same reasoning applies to
$[-R,0]$. Since $\varphi=f+O(\tau)$ for the vast majority of all points
in $B_{R}$, this concludes the proof in the one-dimensional case.

The higher-dimensional case requires only a slight elaboration of this argument.
Let $\rho=R/d$ and choose a linear mapping $L:\reals^d\to\complex$
so that $\varphi(\rho e_j)=L(\rho e_j)$ for each of the unit coordinate vectors $e_j$, $j\in\set{1,2,\cdots,d}$.
Set $\psi=\varphi-L$. Define $\Psi:(\reals^d/\rho\integers^d)\to\complex$ as above. The same reasoning
leads again to the conclusion that $\varphi=L+O(\tau)$ in 
$\set{x: |x_j|<\rho \text{ for all $j\in\set{1,2,\cdots,d}$}}$. Now apply \eqref{phinearlyadditive}
one last time to extend this relation to $B_R$.
\end{proof}

The proof of Proposition~\ref{prop:nearlycharacter} follows that of Proposition~\ref{prop:nearlylinear},
with addition replaced by multiplication wherever appropriate. The details are therefore omitted. \qed

Our application requires the following variant, which is an easy consequence of Proposition~\ref{prop:nearlylinear}.
\begin{lemma} \label{lemma:threefunctions}
Let $R\in(0,\infty)$. Let $\alpha,\beta,\gamma$ be measurable $\complex$--valued functions, and let 
$\frakA,\frakB,\frakC\in\complex$ be nonzero. Suppose that
\begin{equation}
\big|\set{(x,y)\in B_{R}^2: |\frakA\alpha(x)+\frakB\beta(y)+\frakC\gamma(x+y)|>\tau}\big|<\delta |B_{R}^2|.
\end{equation}
Then there exists an affine function $L:\reals^d\to\complex$ such that
\begin{equation}
\big|\set{x\in B_R: |\alpha(x)-L(x)| >C\tau}\big|<\eps(\delta)|B_R|,
\end{equation}
where $\eps(\delta)\to 0$ as $\delta\to 0$.
Here $C$ is a positive constant which depends only on $d,\frakA$.
\end{lemma}

\begin{proof}
By dividing by $\frakA$ and by replacing $\beta$ by $\frakB\frakA^{-1}\beta$
and $\gamma$ by $\frakC \frakA^{-1}\gamma$ we may reduce to the case $\frakA=\frakB=\frakC=1$.
The hypothesis implies that
\[ \big|[\alpha(x)-\alpha(x')] + [\gamma(x+y)-\gamma(x'+y)]\big| \le 2\tau \]
for the vast majority of all $(x,x',y)\in B_{R}^3$.
Therefore
\[ \big|[\alpha(x-y)-\alpha(x'-y)] + [\gamma(x)-\gamma(x')]\big| \le 2\tau \]
for the vast majority of all $(x,x',y)\in B_{R/2}^4$.
Therefore
\[ \big|[\alpha(x-y)-\alpha(x'-y)] - [\alpha(x-y')-\alpha(x'-y')]\big| \le 4\tau \]
for the vast majority of all $(x,x',y,y')\in B_{R/2}^4$.
Substitute $y'=x'-z$ to conclude that
\[ \big|[\alpha(x-y)-\alpha(x'-y)] - [\alpha(x-x'+z)-\alpha(z)]\big| \le 4\tau \]
for the vast majority of all $(x,x',y,z)$ such that $(x,x',y,x'-z)\in B_{R/2}^4$.
Therefore it is  possible to choose $z\in B_{R/16}$ such that this inequality 
holds for the vast majority of all $(x,x',y)\in B_{R/4}^3$.
Defining $\alpha^\natural(x)=\alpha(x-z)-\alpha(z)$, we have 
\[ \big|\alpha^\natural(x-y)-\alpha^\natural(x'-y) - \alpha^\natural(x-x')\big| \le 4\tau \]
for the vast majority of all $(x,x',y)\in B_{R/4}^3$.
Via another substitution we find that
\[ \big|\alpha^\natural(x+y)-\alpha^\natural(x) - \alpha^\natural(y)\big| \le 4\tau \]
for the vast majority of all $(x,x',y)\in B_{R/8}^3$.

By Proposition~\ref{prop:nearlylinear}, there exists a linear function $L$
such that $\alpha^\natural(x)-L(x)=O(\tau)$ for all $x\in B_{R/8}$,
except for a set of measure $<\eta|B_R|$ where $\eta\to 0$ as $\delta\to 0$.
Since $|z|<R/16$, substituting $\alpha(x)=\alpha^\natural(x+z)+\alpha(z)$
gives $\alpha(x)-L_\alpha(x)=O(\tau)$ for the vast majority of all $x\in B_{R/16}$,
for a certain affine function $L_\alpha$.

The same reasoning applies to the functions $\beta,\gamma$, yielding corresponding affine functions $L_\beta,L_\gamma$
which satisfy $L_\alpha(x)+L_\beta(y)+L_\gamma(x+y)=O(\tau)$ on $B_{R/16}$,
and hence which satisfy the same bound on $B_R$, albeit with a larger implicit constant.

One issue remains: It has only been shown that $\alpha=L_\alpha+O(\tau)$ for most
points in $B_{R/16}$, rather than in $B_R$. But
since $\gamma(x+y)+L_\alpha(x)+L_\beta(y)=O(\tau)$ for the vast majority of all $(x,y)\in B_{R/16}^2$,
the same holds for $\gamma(x+y)-L_\gamma(x+y)$.
Thus $\gamma(x)-L_\gamma(x)=O(\tau)$ for the vast majority of all $x\in B_{R'}$ 
for any fixed $R'<R/8$. Take $R'=\frac32\cdot\frac{R}{16}$ for the sake of definiteness.
This reasoning can be repeated with the indices $\alpha,\beta,\gamma$ permuted arbitrarily.
Thus $\alpha(x)-L_\alpha(x)=O(\tau)$ for the vast majority of all $x$ in $B_{R''}$, where $R''=\tfrac32 R'$. 
Repeating this reasoning finitely times demonstrates the stated conclusion, provided that $\delta$ is sufficiently small.
\end{proof}

\section{Extension to higher dimensions}\label{section:extensions}

We next extend Theorem~\ref{thm:main} to $\reals^d$ for arbitrary dimensions $d$, by induction on $d$, 
still considering only nonnegative functions.  Let $B_R=\set{x\in\reals^d: |x|<R}$.

Consider any three Gaussian functions with domain $\reals^1$, denoted by
\begin{align*}
\varphi(s)&= c_p \alpha^{1/2p}\exp(-\alpha (s-a)^2),
\\
\psi(s)&= c_q \beta^{1/2q}\exp(-\beta(s-b)^2),
\\
\xi(s)&= c_r \gamma^{1/2r}\exp(-\gamma(s-c)^2),
\end{align*}
where the normalizing factors $c_p,c_q,c_r$ are chosen so that $\varphi,\psi,\xi$ have norms equal to $1$
in $L^p,L^q,L^r$ respectively, for all parameters $\alpha,\beta,\gamma,a,b,c$.

\begin{lemma} \label{lemma:oneDgaussians}
Let $(p,q,r)\in(1,\infty)^3$ satisfy $p^{-1}+q^{-1}+r^{-1}=2$.
There exists $\Gamma>0$ with the following property.
For any $\eps>0$ there exists $\eta>0$ 
such that for any $\alpha,a,\beta,b$, if $\norm{\varphi*\psi}_{r'}\ge (1-\eta)\Apqr$,
then $|\Gamma-\alpha\beta^{-1}|\le \eps$. Moreover,
if  $\langle \varphi*\psi,\xi\rangle \ge (1-\eta)\Apqr$ then 
\begin{equation}\label{eq:cab1} |c-a-b|<\eps,  \end{equation}
and $(\varphi,\psi,\xi)$ is $o_\eta(1)$--close in $L^p\times L^q\times L^r$
to an exactly extremizing ordered triple of functions.
\end{lemma}
The proof of this elementary fact is left to the reader; alternatively, it is a consequence of
the one-dimensional case Theorem~\ref{thm:main} and the known uniqueness of extremizing
triples up to scalar multiplication and the action of the affine group. \qed

Consider $\reals^{d+1}$, with coordinates $(x,s)\in\reals^d\times\reals^1$.
Let $\delta>0$ be small.
Let $(f,g,h)\in (L^p\times L^q\times L^r)(\reals^d)$ be any $(1-\delta)$--nearly extremizing ordered
triple of nonnegative functions satisfying $\norm{f}_p=\norm{g}_q=\norm{h}_r=1$.

Define \begin{align} F(x)&=\norm{f(x,\cdot)}_{L^p(\reals^1)}
\\
f_x(s)&=f(x,s)/F(x) \text{ if } F(x)\notin \set{0,\infty} \end{align}
and $f_x(s)\equiv 0$ if $F(x)\in\set{0,\infty}$, 
noting that $F$ is finite almost everywhere.
If $F(x)\in(0,\infty)$ then the function $s\mapsto f_x(s)$ has $L^p$ norm equal to $1$. 
Moreover $\norm{F}_{L^p(\reals^d)}=\norm{f}_{L^p(\reals^{d+1})}=1$.
Likewise define $G,g_x,H,h_x$, which have corresponding properties.

We will exploit the representation
\begin{equation} \langle f*g,h\rangle 
= \iint_{\reals^{d+d}} F(x)G(y)H(x+y)\,\langle f_x*g_y,h_{x+y}\rangle\,dx\,dy.  \end{equation}
A first consequence is that
\begin{equation} \langle f*g,h\rangle_{\reals^{d+1}} \le \Apqr \langle F*G,H\rangle_{\reals^{d}}.  \end{equation}

The optimal constant for $\reals^k$ is known to equal $\Apqr^k$, 
where $\Apqr$ is the optimal constant for $\reals^1$.
Since
\begin{align*}
(1-\delta)\Apqr^{d+1}
&\le \langle f*g,h\rangle 
\\
&= \iint_{\reals^{d+d}} F(x)G(y)H(x+y) \langle f_x*g_y,h_{x+y}\rangle \,dx\,dy
\\
&\le \iint_{\reals^{d+d}} F(x)G(y)H(x+y) \Apqr\norm{f_x}_p\norm{g_y}_q\norm{h_{x+y}}_r \,dx\,dy
\\
&\le \Apqr\iint_{\reals^{d+d}} F(x)G(y)H(x+y) \,dx\,dy
\\
&= \Apqr\langle F*G,H\rangle.
\end{align*}
Thus 
\[\langle F*G,H\rangle \ge (1-\delta)\Apqr=(1-\delta)\Apqr\norm{F}_p\norm{G}_q\norm{H}_r,\]
which is to say that $(F,G,H)$ is a $(1-\delta)$--nearly extremizing triple for $\reals^{d}$.
By the induction hypothesis,
$(F,G,H)$ is $o_\delta(1)$--close to some extremizing ordered triple $(F_*,G_*,H_*)$.
By making an affine change of variables in $\reals^d$ we may reduce matters to the case where
\begin{equation*}
F_*(x)=c_p^d \alpha_0^{d/2p}e^{-\alpha_0|x|^2},\ 
G_*(x)=c_q^d \beta_0^{d/2q}e^{-\beta_0|x|^2},\ 
H_*(x)=c_r^d \gamma_0^{d/2r} e^{-\gamma_0|x|^2}
\end{equation*}
where $\alpha_0,\beta_0,\gamma_0$ are fixed constants which depend only on $p,q,r$, 
and $c_p,c_q,c_r$ are the normalizing constants introduced above.

For any $\eps>0$ there exist $\delta_0$ and $R<\infty$ such that
\[\iint_{\reals^{d+d}\setminus B_{R}^2} F(x)G(y)H(x+y)\,dx\,dy<\eps,\]
whenever $(F,G,H)$ is $\delta_0$--close to $(F_*,G_*,H_*)$ in $L^p\times L^q\times L^r$.
Therefore
\begin{align*}
\iint_{\reals^{d+d}\setminus B_{R}^2}\iint_{\reals^{1+1}} f(x,s)&g(y,t)h(x+y,s+t))\,ds\,dt\,dx\,dy
\\
&= \iint_{\reals^{d+d}\setminus B_{R}^2} F(x)G(y)H(x+y)\,\langle f_x*g_y,h_{x+y}\rangle\,dx\,dy
\\
&\le \Apqr \iint_{\reals^{d+d}\setminus B_{R}^2} F(x)G(y)H(x+y)\,dx\,dy
\\
&<\eps\Apqr.\end{align*}
Therefore it suffices to analyze the contribution of those $(x,y)\in B_R^2$
to the integral $\iint_{\reals^{d+d}} F(x)G(y)H(x+y)\,\langle f_x*g_y,h_{x+y}\rangle\,dx\,dy$.

\begin{lemma} \label{lemma:turncrank}
Let $R<\infty$. 
Suppose that $(p,q,r)$ satisfies \eqref{H1},\eqref{H2}.
Let $\delta>0$. If $(F,G,H)$ and $(F_*,G_*,H_*)$ are as in the above discussion, then
there exist measurable subsets $\Omega\subset B_{R}^2$ and $\omega\subset B_{R}$ with the following properties.
Firstly, 
\begin{equation} \iint_{\Omega} F(x)G(y)H(x+y)\,dx\,dy=o_\delta(1) \end{equation}
and
\begin{equation}|\Omega|+|\omega|=o_\delta(1).\end{equation}

Secondly, for $x\notin\omega$,
\begin{equation}\norm{f_x(\cdot)-\varphi_{x}(\cdot)}_{L^p(\reals^1)}=o_\delta(1) \end{equation}
where $\varphi_{x}(s) = c_p \alpha(x)^{1/p}\exp(-\alpha(x)(s-a(x))^2)$.
The functions $g_x,h_x$ are likewise close to corresponding Gaussians $\psi_{x},\xi_{x}$
with parameters $\beta(x),b(x),\gamma(x),c(x)$, on $B_{R}\setminus\omega$.

Thirdly, there exists $(\alpha,\beta,\gamma)\in (0,\infty)^3$ such that
the ordered triple of functions $(e^{-\alpha x^2},e^{-\beta x^2},e^{-\gamma x^2})$
is an extremizing triple for Young's inequality for $\reals^1$ with exponents $(p,q,r)$, which satisfies
\begin{equation} \label{eq:Gexponents} 
|\alpha(x)-\alpha|+|\beta(x)-\beta|+ |\gamma(x)-\gamma| = o_\delta(1) \end{equation}
for all $x\in B_{R}\setminus\omega$. 

Lastly, for all $(x,y)\in B_{R}^2\setminus\Omega$,
\begin{equation}\label{eq:cab} |c(x+y)-a(x)-b(y)| = o_\delta(1).  \end{equation} \end{lemma}
\noindent The constants implicit in the $o_\delta(1)$ notation in these conclusions
are permitted to depend on $R$.

\begin{proof}
It is given that
\begin{multline}
\iint_{\reals^{d}\times\reals^{d}}F(x)G(y)H(x+y)\langle f_{x}*g_{y},h_{x+y}\rangle\,dx\,dy \\
\ge (1-\delta) \iint_{\reals^{d}\times\reals^{d}}F(x)G(y)H(x+y) \Apqr\,dx\,dy.
\end{multline}
On the other hand, for any $(x,y)$, $\langle f_{x}*g_{y},h_{x+y}\rangle\le\Apqr$.  Therefore
\begin{equation} \label{eq:oneDvictory} \langle f_{x}*g_{y},h_{x+y}\rangle \ge (1-\delta^{1/2}) \Apqr\end{equation}
for all $(x,y)\in B_{R}^2\setminus\Omega'$, where $\Omega'\subset B_{R}^2$ is small in the sense that
\begin{equation}\label{eq:Omegaprimebound} \iint_{\Omega'} F(x)G(x)H(x+y)\,dx\,dy\le C\delta^{1/2}.  \end{equation}
Now we may apply the one-dimensional case of our main theorem to conclude that whenever
$(x,y)\in B_{R}^2\setminus\Omega'$, $f_x$ differs from some Gaussian by $o_\delta(1)$
in $L^p$ norm, and likewise $g_y,h_{x+y}$ are close to Gaussians in $L^q,L^r$ norms, respectively.

Later in the argument we will define $\Omega$ to the the union of $\Omega'$ with another set.
Provided that the latter set has measure $o_\delta(1)$, the conclusion \eqref{eq:Omegaprimebound}
remains valid by Young's inequality, since $F_*,G_*,H_*\in L^\infty$
and $(F,G,H) =(F_*,G_*,H_*) + o_\delta(1)$ in $L^p\times L^q\times L^r$.

Again because $(F,G,H) =(F_*,G_*,H_*) + o_\delta(1)$ and because
$F_*(x)G_*(y)H_*(x+y)$ is a strictly positive continuous function, \eqref{eq:Omegaprimebound}
implies that $\Omega'$ is small in the alternative sense that $|\Omega'|=o_\delta(1)$. 
For each $x\in B_{R}$ define
\begin{align}E^x&=\set{y\in B_{R}: (x,y)\notin \Omega'}
\\ \omega&=\set{x\in B_{R}: |E^x| < (1-\delta')|B_{R}|},  \end{align}
where $\delta'$ is a function of $\delta$.
If $\delta'$ is chosen to be a function of $\delta$ which tends to zero 
sufficiently slowly as $\delta\to 0$, then 
\[|\omega| \le o_\delta(1)|B_{R}|,\] 
and uniformly for any two points $x,x'\in B_{R}\setminus\omega$, 
\[|E^x\cap E^{x'}|\ge (1-o_\delta(1))|B_{R}|.\]  
In particular, $E^x\cap E^{x'}$ is nonempty, provided that $\delta$ is sufficiently small.

We have already concluded that whenever $(x,y)\in B_{R}^2\setminus\Omega'$, 
both $f_x,g_y$ are nearly equal to Gaussians. Therefore
for each $x\in B_{R}\setminus\omega$ it is possible to decompose 
$f_x=\varphi_x + \rho_x$ in such a manner that $\norm{\rho_x}_p=o_\delta(1)$,
and $\varphi_x(s)= c_p \alpha(x)^{1/2p}\exp(-\alpha(x) (s-a(x))^2)$. This can be done in a measurable way. 
If $x\in B_{R}\setminus\omega$ and $y\in E^x$ then there is a corresponding decomposition of $g_y$.

Consider any two points $x,x'\in B_{R}\setminus\omega$. There exists $y\in E^x\cap E^{x'}$.
By applying Lemma~\ref{lemma:oneDgaussians} to both pairs $(\varphi_x,\psi_y)$
and $(\varphi_{x'},\psi_y)$, we conclude that
\begin{equation} \left|1-\frac{\alpha(x)}{\alpha(x')}\right|<\eps 
\text{ for all $x,x'\in B_{R}\setminus\omega$.} \end{equation}
Since $\alpha(x)$ is close to $\alpha\ne 0$, and since the same reasoning 
applies to the coefficient functions $\beta,\gamma$, this gives \eqref{eq:Gexponents}.

The final conclusion is that $c(x+y)= a(x)+b(y)+o_\delta(1)$ for the vast majority of all $(x,y)\in B_R^2$.
By the same reasoning as above, we may assume
that for all $z\in B_{2R}$ except for a set $\tilde\omega$ of measure $o_\delta(1)$, 
$h_z=\xi_z + \tau_z$ where $\norm{\tau_z}_r<=o_\delta(1)$ and $\xi_z$ is a Gaussian
$\xi_z(s)= c_r \gamma^{1/2r}\exp(-\gamma(s-c(z))^2)$.
Whenever $(x,y)\in B_R^2$ is such that $x+y\notin\tilde\omega$, 
\[\langle f_x*g_y,h_{x+y}\rangle = \langle \varphi_x*\psi_y,\xi_{x+y}\rangle + o_\delta(1)\]
and therefore $\langle \varphi_x*\psi_y,\xi_{x+y}\rangle\ge 1-o_\delta(1)$.
Consequently by \eqref{eq:cab1}, $c(x+y)=a(x)+b(y)+o_\delta(1)$.
The set of all pairs $(x,y)$ with this property has measure $o_\delta(1)$.
Define $\Omega$ to be its union with $\Omega'$ to conclude the proof.
\end{proof}

By Lemma~\ref{lemma:threefunctions}, there exist an affine function $L$
and a subset $\omega^\dagger$ of $B_R=\set{x\in\reals^d: |x|\le R}$ 
such that $|\omega^\dagger|= o_\delta(1)$ and
\begin{equation} |a(x)-L(x)|=o_\delta(1) \text{ for all $x\in B_R\setminus \omega^\dagger$}.
\end{equation}
Here $C<\infty$ depends only on the dimension $d$.
$L$ may be taken to be real-valued, since $a$ is real.

Since $|\omega^\dagger|=o_\delta(1)$,
$\int_{\omega^\dagger} F_*(x)^p\,dx=o_\delta(1)$.
Since $\norm{F-F_*}_p=o_\delta(1)$, 
\[ \int_{\omega^\dagger} F(x)^p\,dx=o_\delta(1).  \]
Since we have chosen $R$ to be a function of $\delta$ which tends to $\infty$ as $\delta\to 0$,
\[ \int_{\reals^d\setminus B_R} F(x)^p\,dx=o_\delta(1).  \]

Define
$\varphi^\dagger_{x}(s) = c_p \alpha^{1/2p}\exp(-\alpha(s-L(x))^2)$.
Then
\begin{equation} \sup_{x\in B_R\setminus\omega^\dagger}
\norm{\varphi_{x}-\varphi^\dagger_{x}}_p =o_\delta(1) \end{equation}
since $\alpha(x)-\alpha=o_\delta(1)$.
Consider the function $\scriptf(x,s) = \varphi^\dagger_x(s) F_*(x)$, which is a Gaussian 
with domain $\reals^{d+1}$.
\begin{align*}
\norm{f-\scriptf}_p^p
&\le C\int_{\reals^d} F^p(x)\norm{f_x-\varphi_x}_{L^p(\reals)}^p\,dx
\\
&\qquad +  C\int_{\reals^d} F^p(x)\norm{\varphi_x-\varphi^\dagger_x}_{L^p(\reals)}^p\,dx
\\
&\qquad + C\norm{F-F_*}_{L^p(\reals^d)}^p.
\end{align*}

We have already shown that the third term on the right is $o_\delta(1)$.
To analyze the second term, partition $\reals^d$ as 
the union of $\omega^\dagger\cup(\reals^d\setminus B_R)$ and $B_R\setminus \omega^\dagger$.
The contribution of $\omega^\dagger\cup(\reals^d\setminus B_R)$ is $o_\delta(1)$,
because $\norm{\varphi_x-\varphi^\dagger_x}_p$ is uniformly bounded
and $\int_{\omega^\dagger}F^p+ \int_{\reals^d\setminus B_R}F^p=o_\delta(1)$.
The contribution of $B_R\setminus \omega^\dagger$ is $o_\delta(1)$ since 
$a(x)-L(x)=o_\delta(1)$ uniformly for all $x\in B_R\setminus\omega^\dagger$ and $\norm{F}_p=1$.
The first term is treated in the same way, using the facts that
$\norm{f_x-\varphi_x}_{L^p(\reals^1)}=o_\delta(1)$ for all $x\notin\omega$
and $|\omega|=o_\delta(1)$.  We conclude that
\begin{equation} \norm{f-\scriptf}_p=o_\delta(1).  \end{equation}
The same reasoning applies to $g,h$. 

Thus $(f,g,h)$ is close in $L^p\times L^q\times L^r$ norm
to an ordered triple $(\scriptf,\scriptg,\scripth)$,
where $\scriptf(x,s)=F(x)c_p \alpha^{1/2p}e^{-\alpha (s-L(x))^2}$
and $\scriptg,\scripth$ have corresponding structure with exponents $q,r$ and
parameters $\beta,L',\gamma,L''$ respectively.
Here $L,L',L''$ are affine functions.
In particular, each of $\scriptf,\scriptg,\scripth$ is a Gaussian. 
Since $\langle f*g,h\rangle \ge (1-\delta)\Apqr^{d+1}$,
\begin{equation}\label{eq:soclose} \langle \scriptf*\scriptg,\scripth\rangle \ge (1-o_\delta(1))\Apqr^{d+1}
\norm{\scriptf}_p\norm{\scriptg}_q\norm{\scripth}_r.\end{equation}

We are again in the situation of Lemma~\ref{lemma:oneDgaussians}, but in arbitrary dimension.
It is still elementary that \eqref{eq:soclose} implies that the Gaussian triple $(\scriptf,\scriptg,\scripth)$
differs from some extremizing triple by $o_\delta(1)$ in $L^p\times L^q\times L^r$.
This completes the proof of Theorem~\ref{thm:main} for nonnegative functions. \qed

\medskip
Continue to assume that $(p,q,r)$ satisfies \eqref{H1},\eqref{H2}.
Consider any ordered triple of complex-valued functions $(f,g,h)\in (L^p\times L^q\times L^r)(\reals^d)$
which is a $(1-\delta)$--near extremizer for Young's inequality.
Since $|\langle f*g,h\rangle|\le \langle |f|*|g|,|h|\rangle$,
the triple $(F,G,H)=(|f|,|g|,|h|)$ is also a $(1-\delta)$--near extremizer.
Factor $f=e^{i\alpha}F$, $g=e^{i\beta }G$, $h=e^{i\gamma}H$
where $\alpha,\beta,\gamma$ are measurable real-valued functions.
Assume without loss of generality that $\langle f*g,h\rangle$ is real and positive.
Therefore
\begin{multline*}
\iint_{\reals^{d+d}} F(x)G(y)H(x+y)\Re\big(e^{i[\alpha(x)+\beta(y)+\gamma(x+y)]}\big)\,dx\,dy
\\
\ge(1-\delta)
\iint_{\reals^{d+d}} F(x)G(y)H(x+y)\,dx\,dy.
\end{multline*}
We have shown that $(F,G,H)$ differs from some exact extremizer $(\scriptf,\scriptg,\scripth)$
by $o_\delta(1)$ in $L^p\times L^q\times L^r$.
By making an affine change of variables, we may reduce matters to the case where
$(\scriptf,\scriptg,\scripth)$ is fixed.

Fix any $R<\infty$.
Since $\scriptf(x)\scriptg(y)\scripth(x+y)$ is a strictly positive, continuous function,
there exists $\delta_0>0$ such that for all $\delta<\delta_0$,
$|e^{i[\alpha(x)+\beta(y)+\gamma(x+y)]}-1|=o_\delta(1)$
on $B_R^2$, with the exception of a set whose measure is $o_\delta(1)$.
Therefore Lemma~\ref{lemma:threefunctions}
may be invoked to conclude that there exists a real-valued affine function $L$ such that
$\alpha=L+o_\delta(1)$ on $B_R$, with the exception of a set of measure $o_\delta(1)$.
Corresponding conclusions hold for the functions $\beta,\gamma$, with corresponding affine functions
$L',L''$.
Moreover, $L''(x+y)=L(x)+L'(y)+o_\delta(1)$,
in the sense that the difference is an affine function on $\reals^{d+d}$
whose coefficients are $o_\delta(1)$. Thus $(f,g,h)$ has the required structure.
\qed

\end{document}